\documentclass[11pt]{imsart}
\oddsidemargin
-0.0in \topmargin -0.5in
 \usepackage[utf8]{inputenc}

\usepackage{stackengine}

\usepackage[textsize=footnotesize]{todonotes}
\usepackage{xcolor}
\usepackage{hyperref}
\RequirePackage[numbers]{natbib}

\newcommand{\aB}{{\mathcal{D}}}

\usepackage{mathrsfs}
\usepackage{graphicx}
\usepackage{amsmath}        
\usepackage{amssymb}
\usepackage{amsthm}
\usepackage{bm}
\usepackage[mathscr]{euscript}
   
        \usepackage[countmax]{subfloat}  
\usepackage{fourier}
\usepackage{xcolor}
\usepackage{mathtools}

\usepackage{enumerate}
\usepackage{enumitem}
\counterwithin{figure}{section}
\usepackage{natbib}
\usepackage{url}
\usepackage{verbatim}
\usepackage{booktabs}
\usepackage{placeins}
\usepackage{multirow}
\usepackage{caption}

\hypersetup
{
    pdfauthor={Anne van Delft},
    pdfkeywords={functional data, time series, spectral analysis,martingales}
}
\newcounter{alphcount}
{\begin{list}{{\upshape(}\alph{alphcount}\/{\upshape)\ }}%
             {\usecounter{alphcount}\labelwidth1.5em%
              \leftmargin2em\labelsep0.5em\topsep0.25em plus 0.5ex%
              \itemsep0.25em plus 0.5ex\parsep0em}}{\end{list}}
{\begin{list}{{\upshape(#1\arabic{alphcount})\hfill}}%
             {\usecounter{alphcount}\labelwidth2.5em%
              \leftmargin2.5em\labelsep0em\topsep0.25em plus 0.5ex%
              \itemsep0.25em plus 0.5ex\parsep0em}}{\end{list}}
\newcommand\tageq{\addtocounter{equation}{1}\tag{\theequation}}

\newtheorem{thm}{Theorem}[section]
\newtheorem{lemma}{Lemma}[section]    
            
\newtheorem{definition}{Definition}[section]

\newtheorem{Corollary}{Corollary}[section]

\newtheorem{question}{Question}[section]
\newtheorem{remark}{Remark}[section]

\theoremstyle{definition}

\newcommand{\aC}{\mathcal{C}}

\newcommand{\znum}{\mathbb{Z}}

\newcommand{\PH}{PH}
\newcommand{\VR}{\mathrm{VR}}

\allowdisplaybreaks

\providecommand{\AMS}[1]{\textbf{\textit{AMS subject classification: }} #1}

\newcounter{relctr} %% <- counter for relations
\everydisplay\expandafter{\the\everydisplay\setcounter{relctr}{0}} %% <- reset every eq
\AtBeginDocument{}

\begin{document}

\begin{frontmatter}

\title{Measures determined by their values on balls and Gromov-Wasserstein convergence}
\begin{aug}
\author[A]
{\fnms{Anne} \snm{van Delft}\ead[label=e1,mark]{anne.vandelft@columbia.edu}}
\and
\author[B]{\fnms{Andrew J.} \snm{Blumberg}\ead[label=e2]{andrew.blumberg@columbia.edu}}
%%%%%%%%%%%%%%%%%%%%%%%%%%%%%%%%%%%%%%%%%%%%%%
%% Addresses                                %%
%%%%%%%%%%%%%%%%%%%%%%%%%%%%%%%%%%%%%%%%%%%%%%
\address[A]{Department of Statistics, Columbia University, 1255 Amsterdam Avenue, New York, NY 10027, USA.\\ \printead{e1}}
\address[B]{Irving Institute for Cancer Dynamics, Columbia University, 1190 Amsterdam Avenue, New York, NY, 10027, USA.\\ \printead{e2}}

\end{aug}

\begin{abstract}
A classical question about a metric space is whether Borel measures on the space are determined by their values on balls.  We show that for any given measure this property is stable under Gromov-Wasserstein convergence of metric measure spaces.  We then use this result to show that suitable bounded subspaces of the space of persistence diagrams have the property that any Borel measure is determined by its values on balls.  This justifies the use of empirical ball volumes for statistical testing in topological data analysis (TDA).  Our intended application is to deploy the statistical foundations of~\cite{vDB23} for time series of random geometric objects in the context of TDA invariants, specifically persistent homology and zigzag persistence.
\end{abstract}

\AMS{Primary 62R40; secondary 55N31}

\end{frontmatter}

\maketitle 

\section{Introduction}

A classical question in the theory of metric measure spaces is when Borel measures on a metric space $(X,\partial_X)$ are determined by their values on balls.

\begin{question}\label{ques:main}
Let $(X, \partial_X)$ be a metric space and denote $B_\epsilon(x)$ the ball of radius $\epsilon$ with center $x$.  Suppose that $\mu$ and $\nu$ are Borel measures on $X$ such that for all $x \in X$ and $\epsilon > 0$, the equality
\[
\mu(B_\epsilon(x)) = \nu(B_\epsilon(x))
\]
holds.  Is it the case that $\mu = \nu$?
\end{question}

In general, the answer to this question is negative, but over the years many classes of metric spaces for which this property holds have been identified~\cite{Christensen}; notably various kinds of Banach spaces as well as metric spaces that are finite dimensional in a suitable sense, e.g., metric spaces of finite doubling dimension.

A positive answer to this question is salient for the statistical analysis on such Polish-valued processes; many statistical tests for distinguishing measures on metric spaces or, more generally, for characterizing measures can be seen to depend on the values of the measures on balls.  One of the difficulties in developing inference techniques in such a setting stems from the lack of a vector space structure; applying statistical theory requires one to either use notions such as Fr{\'e}chet means or to perform inference on the measure in an indirect manner, such as via the collection of distances. Since Fr{\'e}chet means do not always exist, are not necessarily unique when they do, and are often very expensive to compute, there is considerable interest in the latter approach. 

A focus on distributions of distances is justified by Gromov's reconstruction theorem for metric measure spaces $(S,\partial_S, \mu)$; this result says that the infinite-dimensional distance matrix distribution generated by an iid sample from $\mu$ is a complete invariant (up to measure-preserving isometry) of a metric measure space~\cite{Gromov}.  This theorem implicitly underlies many areas of geometric data analysis (e.g., manifold learning and topological data analysis), and has been used explicitly in the probabilistic analysis of metric measure spaces and their geometry~\cite{BWM18, Brecheteau, Greven, Ozawa}.  We are particularly interested in applications for the more general type of processes described in~\cite{vDB23}, which introduce a new framework for statistical inference methods for the analysis of nonstationary processes of geometric objects. In Section 3 of \cite{vDB23}, the authors provided an extension of Gromov's mm reconstruction theorem and proved that the infinite-dimensional distance matrix distribution
\[
(\partial_S(X_t,X_s) : t,s \in \znum)  
\]
is a complete invariant of a stationary ergodic Polish-valued stochastic process $X=(X_t\colon t\in \znum)$. Put differently, that it is a complete invariant of a metric measure-preserving dynamical system $(S^\znum, \rho_S, \mu_X, \theta)$ where the measure-preserving map $\theta$ is ergodic under $\mu_X$.

However, one sample of the distance matrix distribution does not give us enough degrees of freedom for statistical inference. Therefore, the authors proposed instead to look at the corresponding ball volume processes. Building on a result in~\citep[][thm 2.2]{BWM18}, they proved that the law of \textit{any} Polish-valued stochastic process can be characterized by the corresponding family of ball volume processes under the assumption that the law is a doubling measure, which was subsequently exploited to detect nonstationary laws in stochastic processes due to time-changing shape dynamics.  More specifically, define the finite-dimensional distribution (fidi) of the process $X$ on the time index set $J$ by $\mu^X_{J}$, and let $\rho_J$ be a suitable metric that metrizes the product topology on the corresponding product.  The result is given here for completeness.

\begin{thm}{\citep[][Thm 3.6]{vDB23}}\label{thm:charmux}
Let $(X_t: t\in  \znum)$ be a Polish-valued stochastic process with law $\mu_{X}$ a locally finite Borel regular doubling measure. 
Then 
 the process $\mu^X_{J}\big( B_\epsilon(\pi_{J} \circ X): \epsilon \ge 0\big)$, where $B_\epsilon(x)$ denotes the ball of radius $\epsilon$ with center $x$, characterizes the $|J|$-dimensional dynamics on the time set $J$.
 \end{thm}

The assumption that $\mu_X$ is doubling ensures that the measure is completely determined by its values on metric balls.  The property of being determined by the value on balls significantly simplifies statistical analysis over the use of the infinite-dimensional matrix distributions since empirical ball volume processes take the form of $U$-processes \citep{vDB23}. \\

This simplification motivates the current paper.  The power of topological data analysis comes in large part from the fact that its invariants live in metric spaces.  Specifically, persistent homology and its variants take values in the space of barcodes (also referred to as persistence diagrams), which can be given metric space structures using various matching distances (e.g., see~\cite{TDAbook, c14, c09, o15, eh08, eh10}).  We would like to know that measures on barcode space are determined by their values on metric balls.  However, barcode space is not a doubling metric space; instead, it is infinite dimensional.  Moreover, it is not a Banach space.  On the other hand, Sheehy and Sheth have recently shown that any of the spaces of bounded barcodes with a bounded number of points are $\epsilon$-close in the Gromov-Hausdorff metric to a doubling metric space~\cite{SS22}.  In particular, they are $\epsilon$-close to metric spaces that have the property that any Borel measure is determined by its values on balls for any $\epsilon$.  Put another way, there is a sequence $\{B_n\}$ that satisfies this property for each $B_n$ which converges to $B$ in the Gromov-Hausdorff metric.  This motivates our first main theorem.

\begin{thm}\label{thm:mainthm}
Assume that $(X, \partial_X, \mu_X)$ and $\{(X_n, \partial_{X_n}, \mu_{X_n})\}$ are compact metric measure spaces. Suppose that $\{X_n\}$ converges to $X$ in the Gromov-Wasserstein metric, and that for each $X_n$, $\mu_{X_n}$ is determined by its values on balls.  Then $\mu_X$ is determined by its values on balls.
\end{thm}

Combining this with the result of~\cite{SS22}, we obtain the following essential application:

\begin{thm}\label{thm:barcods}
The space $\aB^N_\alpha$ of $N$-point bounded barcodes has the property that any Borel measure is determined by its values on balls.
\end{thm}

Theorem~\ref{thm:barcods} has two essential corollaries that are essential for applications in topological data analysis.  In the following statement, we are imagining that we have a distribution on point clouds consisting of $n$ points sampled from a metric measure space $(X, \partial_X, \mu_X)$.  A typical way to produce such a distribution on point clouds is to consider taking $n$ iid samples from $\mu_X$ (see \autoref{sec4}).  The persistent homology functor $PH_k$ takes a point cloud to the composite of the Vietoris-Rips complex and the homology $H_k$.

\begin{Corollary}\label{thm:PH}
The pushforward of any Borel measure on point clouds under the persistent homology functor $PH_k$ is determined by its values on balls.
\end{Corollary}

Persistent homology arises from filtered simplicial complexes or spaces; i.e., systems $A_0 \to A_1 \to A_2 \to \ldots \to A_\ell$.  In fact, it turns out that a variant of persistent homology exists which allows the arrows to go in either direction, known as zigzag persistence~\cite{csm09}.  We explain below in more detail how zigzag persistence arises, but roughly speaking we have the following corollary.

\begin{Corollary}\label{thm:zigzag}
The pushforward of any Borel measure on point clouds under the zigzag persistent homology functor is determined by its values on balls.
\end{Corollary}

The paper is organized as follows.  In section~\ref{sec:mainthm}, we prove Theorem~\ref{thm:mainthm}.  The application to the space of (bounded) barcodes is discussed in Section~\ref{sec3}, and in Section~\ref{sec4} we explain consequences for topological data analysis, specifically persistent homology and zigzag persistence.

\FloatBarrier
\section*{Acknowledgements}
We would like to thank Mike Lesnick and Mike Mandell for useful conversations.  Both authors were partially supported by NSF grant DMS-2311338.  Blumberg was partially supported by ONR grant N00014-22-1-2679.  

\section{Convergence of measures and values on balls}\label{sec:mainthm}

The purpose of this section is to prove Theorem~\ref{thm:mainthm}, which shows that if a measure can be approximated in a suitable sense by measures determined by their values on balls, it itself is determined by its values on balls.  Throughout, let $(X, \partial_X, \mu_X)$ be a metric measure space.  We begin by reviewing the Caratheodory construction of an outer measure given the measures of a family of subsets of $X$.

\begin{definition}
Let $\aC$ denote a collection of subsets of $X$.  We define an outer
measure $\mu_C$ on $X$ by the formula
\[
\mu_C(A) = \inf_{\{C_i\} \subset \aC, A \subset \bigcup_i C_i} \sum_i \mu(C_i),
\]
where we require $\{C_i\}$ to be a countable set.
\end{definition}

We now define a new Borel measure on $X$ using the values of $\mu$ on
balls as follows:

\begin{definition}
Let $\mu^*$ be the Borel measure associated to the metric outer
measure
\[
\mu^*(A) = \sup_{\epsilon > 0} \mu_{B_{\epsilon}}(A).
\]
\end{definition}

\begin{remark}
Notice that for $\epsilon > \epsilon'$,
$\mu_{B_{\epsilon'}}(A) \geq \mu_{B_{\epsilon}}(A)$.
\end{remark}

Now, it is always true that $\mu(A) \leq \mu^*(A)$; this follows by
approximating $A$ by a cover of metric balls.  Thus, the interesting
question is whether $\mu(A) \geq \mu^*(A)$.\\

In order to prove that Borel measures on barcode spaces are determined by their values on metric balls, we require a result regarding convergence in the Gromov-Wasserstein metric of a sequence of compact metric measure spaces (\autoref{thm:convergence}). To make this precise we recall some basic facts and definitions. First, we mention the well-known fact that the topology of weak convergence of probability measures on a separable metric space $(Z,\partial_Z)$ is generated by the L{\'e}vy-Prokhorov metric $d_P$. Provided that the space has bounded diameter, the same holds true for the Wasserstein metric $d_W$, which is given by
\[
d_W(\mu,\nu):= \inf_{\pi \in \Pi(u,\nu)} \int_{Z\times Z} \partial_Z(x,x^\prime) d \pi(x,x^\prime),
\]
where the infimum is taken over all couplings $\pi$ of $\mu$ and $\nu$. Indeed, it can be shown that  \citep{Gibbs}
\[
d_P^2 \le d_W \le (\text{diam}(Z)+1)d_P. \tageq \label{eq:wassersandwich}
\]
In analogy to the Levy-Prokhorov metric giving rise to the Gromov-Prokhorov metric, 
the Wasserstein metric gives rise to the Gromov-Wasserstein metric between two metric measure spaces $\mathcal{X}=(X, \partial_X,\mu_X)$ and $\mathcal{Y}=(Y, \partial_Y, \mu_Y)$; 
\[
d_{GW} =\inf_{(i_X, i_Y, Z)} d_W^{(Z,\partial_Z)} \big((i_X)_{\ast}\mu_X,(i_Y)_{\ast}\mu_Y\big)
\]
where the infimum is taken over all isometric embeddings $i_X\colon \text{supp}(\mu_X) \to Z$ and $i_Y\colon \text{supp}(\mu_Y) \to Z$. 

Key in our argument is then the following theorem.

\begin{thm}\label{thm:convergence}
Assume that $(X, \partial_X, \mu_X)$ and $\{(X_n, \partial_{X_n}, \mu_{X_n})\}$ are compact metric measure spaces.  Suppose that $\{X_n\}$ converges to $X$ in the Gromov-Wasserstein
metric, and that for each $X_n$, $\mu_{X_n} = \mu^*_{X_n}$.  Then
$\mu_X = \mu^*_X$.
\end{thm}

The following lemma, which follows immediately from \eqref{eq:wassersandwich} and from Lemma 5.7 in \cite{Greven} (which proves the analogous result for the Gromov-Prokhorov distance), allows us to convert the statement above to a question involving Wasserstein convergence.

\begin{lemma}
Let $\{(X_n, \partial_{X_n}, \mu_X)\}$ be a sequence of compact metric measure spaces.  Then
$\{X_n\}$ converges to $X$ in the Gromov-Wasserstein metric if and
only if there exists a metric measure space $(Z, \partial_Z, \mu_Z)$ such that there exist
isometric embeddings $i_n \colon X_n \to Z$ and $i \colon X \to Z$
such that $\{i_n^* \mu_{X_n}\}$ converges to $i^* \mu_X$ in the
Wasserstein metric.
\end{lemma}

Therefore, to prove Theorem~\ref{thm:convergence} it suffices to consider the weak convergence of $i_n^* \mu_{X_n}$ to $i^* \mu_X$ on the fixed metric measure space $Z$. 

\begin{proof}[Proof of \autoref{thm:convergence}]

Since compactness implies the Wasserstein distance metrizes weak convergence, we have for any continuity set $A$, 
\[
\mu^*_{X_n}(A) \to \mu^*_X(A),
\]
by the Portmanteau lemma.  Moreover, $\mu^*_{X_n}(A) = \mu_{X_n}(A)$
by hypothesis.  Therefore, we have that $\mu_{X_n}(A) \to \mu^*_X(A)$.
This implies immediately that for continuity sets $A$, $\mu^*_X(A) =
\mu_X(A)$. The result now follows from \autoref{lem:sepclass} below.
\end{proof}
\begin{lemma}\label{lem:sepclass}
Let $(X, \partial_X, \mu_X)$ be a metric measure space where
$(X, \partial_X)$ is Polish and $\mu_X$ a nonnegative Borel measure.  The collection $\Gamma_\mu$ of $\mu_X$-continuity sets forms a separating class. 
\end{lemma}

\begin{proof}
It suffices to show that $\Gamma_\mu$ is a $\pi$-system that generates the Borel $\sigma$-algebra $\mathcal{B}(X)$. Clearly, for $A, B \in \Gamma_{\mu}$, $\partial(A \cap B) \subset \partial A \cup \partial B$ and so 
\[
\mu_X(\partial(A \cap B)) \leq \mu_X(\partial A \cup \partial B) \leq \mu_X(\partial A) + \mu_X(\partial B) = 0,
\]
showing that $\Gamma_\mu$ is a $\pi$-system. 
To show that it generates the Borel $\sigma$-algebra on $X$, we remark that (by assumption) $X$ is completely regular and $\mu_X$ a nonnegative Borel measure. The set $\Gamma_\mu$ can therefore be shown to form a base for the topology of $X$ \citep[proposition 8.2.8, ][]{BogII}. Since $X$ is second countable, $\Gamma_\mu$ has a countable subfamily that is also a base for its topology, which completes the proof.
\end{proof}

\section{Any measure on barcode space is determined by its value on balls} \label{sec3}

Our main application is to the space $\aB^N_\alpha$ of $N$-point bounded persistence diagrams; these are persistence diagrams where the points are contained in a box $[0,\alpha] \times [0,\alpha]$ and there are at most $N$ points in each diagram.

Sheehy and Sheth observe that although $\aB^N_\alpha$ does not have finite doubling dimension, for any $\epsilon$ it is close to a metric space of finite doubling dimension in the following sense.

\begin{thm}
Fix $\epsilon > 0$.  There exists a metric space $\aB_\epsilon$ such that 
\[
d_{GH}(\aB_\epsilon, \aB^N_\alpha) < \epsilon.
\]
In particular, we have that for any sequence of indices $\{\epsilon_i\}$ such that $\epsilon_i \to 0$, the sequence $\{\aB_{\epsilon_i}\}$ converges to $\aB^N_\alpha$ in the Gromov-Hausdorff metric.
\end{thm}

We utilize the fact that $\aB_\epsilon$ surjects onto $\aB^N_\alpha$; to be more specific, whereas $\aB^N_\alpha$ is the quotient of the upper half plane by the line $y = x$, $\aB_\epsilon$ is the quotient by an $\epsilon$-dense sample of the line $y = x$.  Therefore, we have an evident quotient map $p_\epsilon \colon \aB_\epsilon \to \aB^N_\alpha$.
Because $p_\epsilon$ is a surjection, given any Borel measure on $\aB^N_\alpha$ we can pull it back along $p_\epsilon$ to obtain a measure on $\aB_\epsilon$.

\begin{lemma}
Let $\mu$ be a Borel measure on $\aB^N_\alpha$.  Fix a sequence of indices $\{\epsilon_i\}$ such that $\epsilon_i > 0$ and $\epsilon_i \to 0$ as $i \to \infty$.  Then the sequence $\{(\aB_{\epsilon_i}, p_{\epsilon_i}^* \mu)\}$ converges to $(\aB^N_\alpha, \mu)$ in the Gromov-Wasserstein metric.
\end{lemma}

\begin{proof}
Recall that convergence in the Gromov-Wasserstein metric is equivalent to convergence of the infinite distance matrix distributions; this is a consequence of Gromov's reconstruction theorem, see e.g.~\cite[Thm 9.5]{Greven} for a statement, together with \eqref{eq:wassersandwich}.  Let $B_1$ and $B_2$ be barcodes in $B_\epsilon$ with no points within $\epsilon$ of the line $y=x$; then \[
d_{\aB^N_\alpha}(B_1, B_2) = d_{\aB_\epsilon}(p_\epsilon B_1, p_\epsilon B_2).
\]
If at least one of the pair $B_1$ and $B_2$ has points within $\epsilon$ of the line $y = x$, then we have that
\[
|d_{\aB^N_\alpha}(B_1, B_2) - d_{\aB_\epsilon}(p_\epsilon B_1, p_\epsilon B_2)| \leq N\epsilon.
\]
Since the area of this strip is bounded above by $\sqrt{\alpha} \epsilon$, it is clear that the distance matrix distributions converge as $\epsilon \to 0$.
\end{proof}

As a corollary of the main theorem, we arrive at the following theorem of interest in intended applications:

\begin{thm}\label{thm:barcodeproperty}
The space $\aB^N_\alpha$ of $N$-point bounded barcodes has the property that any Borel measure is determined by its values on balls.
\end{thm}

\section{Applications to persistent homology and to zigzag persistence} \label{sec4}

The real importance of Theorem~\ref{thm:barcodeproperty} is for studying distributions on barcode space induced by pushforward along the functors of topological data analysis. To make this more precise, consider a Polish-valued stochastic process $X=(X_t: t\in \znum)$ with law $\mu_X$, which can be viewed as a metric measure dynamical system \citep[see ][section 3]{vDB23}. For a finite sample $\{{X}_t\}_{t=1}^J$ we denote the  joint distribution on the index set $[J]=\{1,\ldots, J\}$ by  $\mu_J^X$. In applications, we only observe the finite sample $\{{X}_t\}_{t=1}^J$ in the form of a sequence of  point clouds $\tilde{X}_1,\ldots, \tilde{X}_J$, each consisting of an $n$-dimensional sample of its underlying latent metric measure space. Observe that $\{{X}_t\}_{t=1}^J$ can be viewed as a metric measure space $(S^{|J|}, \rho_{|J|}, \mu_J^X)$. A similar statement holds true for the corresponding sequence of point clouds. For example, the independent sampling case from a fixed finite measure  metric space -- which has received paramount attention in the TDA literature -- is a special case where then simply $\mu_n^{{X}} = (\mu_1^{{X}})^{\otimes n}$. Note this applies in particular to random sampling from $\mathbb{R}^n$. 

The most well-known functor of topological data analysis is persistent homology.  The joint measure $\mu_J^{{X}}$ can be pushed forward by the persistent homology functor to obtain a distribution on barcodes, i.e., $(\PH_k)_{\star} \mu_J^{{X}}$. From 
\autoref{thm:barcodeproperty} it follows immediately that the resulting distribution on barcodes is determined by its values on balls. 

\begin{Corollary}
The pushforward of any Borel measure on point clouds under the persistent homology functor $PH_k$ is determined by its values on balls.
\end{Corollary}

 ``Classical'' persistent homology as an intermediate step involves filtered sequences of spaces of the form 
 \[
\VR_{\epsilon_0}(\tilde{X}) \to \VR_{\epsilon_1}(\tilde{X})  \to \ldots \to \VR_{\epsilon_n}(\tilde{X}) 
\]
for $\epsilon_0 < \epsilon_1 < \ldots < \epsilon_n$ and 
where $\VR$ denotes the Vietoris-Rips complex (though other complexes could be considered), and where $\tilde{X}$ is a given point cloud.  Here the arrows are inclusions of complexes induced by the functoriality of the Vietoris-Rips complex.  In such a diagram, all of the arrows go in the same direction.  Applying homology $H_k(-)$ to this sequence then produces the persistent homology of $\tilde{X}$, which can be summarized by a persistence diagram and denoted $\PH_k(\tilde{X})$.

A variant of this setup, referred to as ``zigzag persistence'', relaxes the requirement that the arrows go in the same direction.  A notable example is defined as follows.  Suppose that we have a sample of point clouds $\tilde{X}_1, \tilde{X}_2, \ldots, \tilde{X}_J$, obtained as a sample from the metric measure space $(S^{|J|}, \rho_{|J|}, \mu_J^X)$ as above.  Then there is a zigzag diagram constructed by taking the Vietoris-Rips complexes for unions of adjacent indices (imposing a fixed scale $\epsilon$ throughout the diagram):
\[
\VR_\epsilon(\tilde{X}_1) \rightarrow \VR_{\epsilon}(\tilde{X}_1 \cup \tilde{X}_2) \leftarrow \VR_{\epsilon}(\tilde{X}_2) \rightarrow \VR_{\epsilon}(\tilde{X}_2 \cup \tilde{X}_3) \rightarrow \ldots \tageq \label{eq:zzdiag}
\]
Note that this diagram depends on the choice of indices, i.e., the ordering of the point clouds.  Applying homology $H_k(-)$ to this diagram produces a zigzag of homology groups that can again be summarized as a barcode, the zigzag persistence $ZP_k$.  We now get a measure on barcodes by $({ZP}_k)_\ast \mu_J^{{X}}$, and so for instance constructing a sample from $\mu_X$ on a given index set and pushing forward along the zigzag persistence functor again yields a distribution on barcodes that is determined by its values on balls.  However, note that there are a wide variety of other possible zigzag diagrams could be constructed from a given data set. Although it is therefore challenging to state a precise general theorem, a consequence of \autoref{thm:barcodeproperty} is that for any reasonable procedure that produces zigzag persistence diagrams, the pushforward to barcode space induces a distribution that is determined by its values on balls. 

\begin{Corollary}
The pushforward of any Borel measure on point clouds under the zigzag persistent homology functor $ZP_k$ is determined by its values on balls.
\end{Corollary}

\end{document}